\documentclass[12pt]{article}
\usepackage[usenames,dvipsnames]{xcolor}

\usepackage[affil-it]{authblk}
\usepackage{amsmath,amssymb,amsfonts,amsthm,pgf,tikz}
\usetikzlibrary{shapes,positioning,decorations}
\usetikzlibrary{decorations.pathreplacing}
\usepackage[bookmarks=true,pdfborder={0 0 0}]{hyperref}
\usepackage{soul}
\usepackage{enumerate} 
\usepackage{bm} 

\newtheorem{theorem}{Theorem}[section]
\newtheorem{lemma}[theorem]{Lemma}
\newtheorem{corollary}[theorem]{Corollary}

\newtheorem{obs}[theorem]{Observation}

\theoremstyle{definition}
\newtheorem{definition}[theorem]{Definition}
\newtheorem{example}[theorem]{Example}

\theoremstyle{remark}
\newtheorem{remark}[theorem]{Remark}

\numberwithin{equation}{section}
\DeclareMathOperator{\rank}{rank}

\DeclareMathOperator{\n}{N}

\DeclareMathOperator{\match}{match}
\DeclareMathOperator{\MR}{\rm{MR}^{--}}

\DeclareMathOperator{\jac}{Jac}

\DeclareMathOperator{\im}{Im}
\DeclareMathOperator{\all1}{\bf 1}
\DeclareMathOperator{\at}{\;{\rule[-3.6mm]{.1mm}{8mm}}}

\newcommand{\floor}[1]{\Big\lfloor #1 \Big\rfloor}
\renewcommand{\i}{\mathrm{i}}
\newcommand{\seq}[3]{#1_{1} #3 #1_{2} #3 \dots #3 #1_{#2}} 
\newcommand{\set}[1]{\{#1\}}
\renewcommand{\d}{\text{d}}

\begin{document}
\title{Spectral characterization of matchings in graphs}
\author{Keivan Hassani Monfared \thanks{University of Calgary, k1monfared@gmail.com} \footnote{The work of this author was partially supported by the Natural Sciences and Engineering Research Council of Canada.}
\and Sudipta Mallik \thanks {Department of Mathematics \& Statistics, Northern Arizona University, 805 S. Osborne Dr. PO Box: 5717, Flagstaff, AZ 86011, USA, sudipta.mallik@nau.edu}
}
\date{}
\maketitle   

\renewcommand{\thefootnote}{\fnsymbol{footnote}} 
\footnotetext{\emph{2010 Mathematics Subject Classification. 05C50,65F18\\ Keywords: Skew-Symmetric Matrix, Graph, Tree, Matching, The Jacobian Method, Spectrum, Structured Inverse Eigenvalue Problem.}}    
\renewcommand{\thefootnote}{\arabic{footnote}} 
\begin{abstract}
A spectral characterization of the matching number (the size of a maximum matching) of a graph is given. More precisely, it is shown that the graphs $G$ of order $n$ whose matching number is $k$ are precisely those graphs with the maximum skew rank $2k$ such that for any given set of $k$ distinct nonzero purely imaginary numbers there is a real skew-symmetric matrix $A$ with graph $G$ whose spectrum consists of the given $k$ numbers, their conjugate pairs and $n-2k$ zeros.
\end{abstract}

\section{Introduction}
A {\it matching} in a graph $G$ is a set of vertex-disjoint edges. A {\it maximum matching} in $G$ is a matching with the maximum number of edges among all matchings in $G$.  A {\it perfect matching} in a graph $G$ on $n$ vertices is a maximum matching consisting of $\frac{n}{2}$ edges.  Matchings are well-studied combinatorial objects with practical applications such as Hall's marriage theorem (1935). For a full treatment of matchings see \cite{LP}. In 1947 Tutte gave necessary and sufficient conditions for a graph to have a perfect matching. 

\begin{theorem}\cite{tutte}\label{tutte}
A graph $G$ has a perfect matching if and only if for each vertex subset $S$ of $G$, the number of odd components of $G-S$ is at most $|S|$.
\end{theorem}

The matching number, denoted by $\match(G)$, of a graph $G$ is the number of edges in a maximum matching in $G$. So Theorem \ref{tutte} characterizes all graphs $G$ on $n$ vertices with $\match(G)=\frac{n}{2}$. In this article we give another set of necessary and sufficient conditions for a graph $G$ to have a perfect matching. These conditions concern eigenvalues of skew-symmetric matrices corresponding to $G$. For a given positive integer $k$, we also give necessary and sufficient conditions for a graph $G$ to have $\match(G)=k$.

We begin by introducing some required terminology as given in \cite{HmM}. Let $A = [a_{ij}]$ be an $n\times n$ real skew-symmetric matrix. The {\it order} of $A$ is  $n$, and we denote it by $|A|$. The {\it graph of $A$}, denoted by $G(A)$, has the vertex set $\{1,2,\ldots,n\}$ and the edge set $\{ \{i,j\}: a_{ij}\neq 0, 1\leq i<j\leq n\}$. The set $S^{-}(G)$ denotes the set of all real skew-symmetric matrices whose graph is $G$. The {\it maximum skew rank} of $G$, denoted by $\MR (G)$, is  defined to be $\max\{\rank(A):A\in S^{-}(G)\}$. The maximum skew rank and the matching number of a graph are related as follows.

\begin{theorem}\cite[Theorem $2.5$]{skewrank}\label{maxskewrank}
$\MR (G)=2\match(G)$ for all graphs $G$.
\end{theorem}

The rank of a real symmetric or skew-symmetric matrix can be determined by its nonzero eigenvalues as follows.

\begin{lemma}\cite[Corollary 2.5.14]{HJ}\label{nonzeroeigen}
Let $A$ be a real symmetric or skew-symmetric matrix. Then $\rank(A)$ equals to the number of nonzero eigenvalues of $A$.
\end{lemma}

A {\it full matching} in a graph $G$ on $n$ vertices is a matching $M$ such that $2|M|=n$ or $n-1$, i.e., $\match(G)=\lfloor \frac{n}{2}\rfloor$.  In Section 2 we determine existence of a full matching of $G$ using nonzero eigenvalues of matrices in  $S^{-}(G)$. In Section 3, for a given positive integer $k$, we give necessary and sufficient conditions for $G$, in terms of nonzero eigenvalues of matrices in  $S^{-}(G)$, to have $\match(G)=k$.

To study matchings in connected graphs we first study matchings in trees. A certain kind of trees called NEB trees is introduced in \cite{HmM} and it has been shown that any NEB tree has a full matching. We introduce required definitions and notation for NEB trees as given in \cite{HmM}.

\textbf{Notation:} Let $T$ be a tree, and let $T(v)$ denote the forest obtained from $T$ by deleting vertex $v$. Also, let $T' = T_w(v)$ denote the connected component of $T(v)$ that contains the neighbor $w$ of $v$. $T'$ is a tree, hence it makes sense to consider $T'(w) = \left(T_w(v)\right)(w)$, the forest obtained from $T'$ by deleting vertex $w$, and $T'' = \left(T_w(v)\right)_u(w)$, the connected component of $T'(w)$ that contains the neighbor $u$ of $w$, and so on. For simplicity, we denote the tree $(\cdots(((T_{v_2}(v_1))_{v_3}(v_2))_{v_4}(v_3))\cdots)_{v_k}(v_{k-1})$ by $T_{v_k}(v_1,v_2,\ldots,v_{k-1})$, and the forest $((\cdots(((T_{v_2}(v_1))_{v_3}(v_2))_{v_4}(v_3))\cdots)_{v_k}(v_{k-1}))(v_k)$ by $T(v_1,\ldots,v_k)$. See Figure \ref{FigGenTree}. 

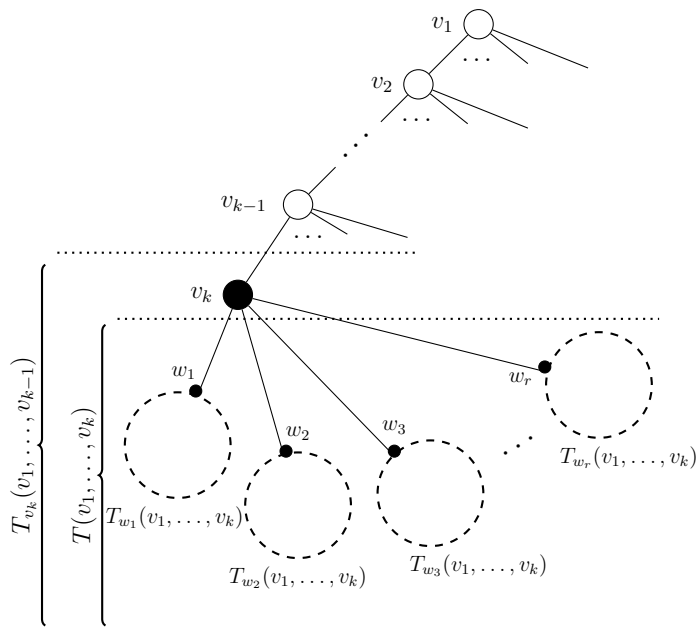
\begin{figure}[h]
\begin{center}
\begin{tikzpicture}[scale=.8]
	\node[draw, circle] (00) at (0,0) {};
	\node[,scale=.8] () at (-.6,0) {$v_1$};
	\node[draw, circle] (11) at (-1,-1) {};
	\node[,scale=.8] () at (-1.6,-1) {$v_2$};
	\node[,scale=.8] (12) at (0,-.6) {$\cdots$};
	\node[] (13) at (1,-.8) {};
	\node[] (14) at (2,-.8) {};
	\node[rotate=43.5] (22) at (-2,-2) {$\cdots$};
	\node[,scale=.8] (23) at (-1,-1.6) {$\cdots$};
	\node[] (24) at (0,-1.8) {};
	\node[] (25) at (1,-1.8) {};
	\node[draw, circle] (33) at (-3,-3) {};
	\node[,scale=.8] () at (-3.9,-3) {$v_{k-1}$};

	\node[] (35) at (-1,-2.7) {};
	\node[] (36) at (0,-2.7) {};
	
	\node[draw, circle, fill] (44) at (-4,-4.5) {};
	\node[,scale=.8] () at (-4.6,-4.5) {$v_k$};
	\node[,scale=.8] (45) at (-2.8,-3.55) {$\cdots$};	
	\node[] (46) at (-2,-3.6) {};
	\node[] (47) at (-1,-3.6) {};
	
	\node[draw, thick, circle, dashed, minimum size = 40] (55) at (-5,-7) {};
	\draw[fill] (-4.7,-6.1) circle (.1);
	\node[scale=.7] () at (-4.9,-5.8) {$w_1$}; 
	\node[scale=.7] () at (-5.05,-8.2) {$T_{w_1}(v_1,\ldots,v_k)$};

	\node[draw, thick, circle, dashed, minimum size = 40] (56) at (-3,-8) {};
	\draw[fill] (-3.2,-7.1) circle (.1);
	\node[scale=.7] () at (-3,-6.8) {$w_2$};
	\node[scale=.7] () at (-3,-9.2) {$T_{w_2}(v_1,\ldots,v_k)$};

	\node[draw, thick, circle, dashed, minimum size = 40] (58) at (-.8,-7.8) {};
	\draw[fill] (-1.4,-7.1) circle (.1);
	\node[scale=.7] () at (-1.4,-6.7) {$w_3$};
	\node[scale=.7] () at (0,-9) {$T_{w_3}(v_1,\ldots,v_k)$};

	\node[rotate=30] (57) at (.7,-7) {$\cdots$};

	\node[draw, thick, circle, dashed, minimum size = 40] (59) at (2,-6) {};
	\draw[fill] (1.1,-5.7) circle (.1);
	\node[scale=.7] () at (.7,-5.9) {$w_r$};
	\node[scale=.7] () at (2.5,-7.2) {$T_{w_r}(v_1,\ldots,v_k)$};
	
	\draw[] (00) -- (11) -- (22) -- (33) -- (44) -- (55);
	\draw[] (00) -- (13);
	\draw[] (00) -- (14);
	\draw[] (11) -- (24);
	\draw[] (11) -- (25);
	\draw[] (33) -- (46);
	\draw[] (33) -- (47);
	
	\draw[] (44) -- (56);
	\draw[] (44) -- (59);
	\draw[] (44) -- (58);
	
	\draw[decorate,decoration={brace},thick] (-6.2,-10) to node[rotate=90,midway,above,scale=.8] (bracket) {$T(v_1,\ldots,v_k)$} (-6.2,-5);
	\draw[thick, dotted] (-6,-4.9) -- (3,-4.9);
	
	\draw[decorate,decoration={brace},thick] (-7.2,-10) to node[rotate=90,midway,above,scale=.8] (bracket) {$T_{v_k}(v_1,\ldots,v_{k-1})$} (-7.2,-4);
	\draw[thick, dotted] (-7,-3.8) -- (-1,-3.8);
\end{tikzpicture}
\caption{Tree $T$ with subgraphs $T(v_1,\ldots,v_k)$ and $T_{v_k}(v_1,\ldots,v_{k-1})$.}\label{FigGenTree}
\end{center}
\end{figure}

\begin{example}
Consider the graph in Figure \ref{FigConcreteTree}. Delete vertex $1$ and consider the connected component that contains the vertex $2$. This tree is denoted by $T_2(1)$. Then in this tree delete vertex $2$. The obtained forest is denoted by $T(1,2)$. The connected component of $T(1,2) $ that contains the vertex $3$ is denoted by $T_3(1,2)$.

\begin{figure}[h]
\begin{center}
\begin{tikzpicture}[scale=.8]
	\node[circle, fill, gray] (1) at (0,0) {};
	\node[circle, draw, black] () at (0,0) {};
	\node[,scale=.8] () at (-.6,0) {$1$};
	
	\node[circle, fill, gray] (2) at (-1,-1) {};
	\node[draw, black, circle] () at (-1,-1) {};
	\node[,scale=.8] () at (-1.6,-1) {$2$};
	
	\node[draw, circle] (6) at (1,-1) {};
	\node[,scale=.8] () at (1.4,-1) {$6$};
	
	\node[draw, circle] (4) at (-2,-2) {};
	\node[,scale=.8] () at (-2.6,-2) {$4$};
	
	\node[draw, circle] (3) at (0,-2.2) {};
	\node[,scale=.8] () at (0.6,-2.2) {$3$};
	
	\node[draw, circle] (5) at (1,-3.4) {};
	\node[,scale=.8] () at (1.6,-3.4) {$5$};
	
	\draw[] (4) -- (2) -- (1) -- (6);
	\draw[] (2) -- (3) -- (5);
	
	\draw[decorate,decoration={brace},thick] (-.8,-3.7) to node[rotate=90,midway,above,scale=.8] (bracket) {$T_3(1,2)$} (-.8,-1.8);
	\draw[thick, dotted] (-.7,-1.8) -- (1.7,-1.8);
	
	\draw[decorate,decoration={brace},thick] (-2.8,-3.7) to node[rotate=90,midway,above,scale=.8] (bracket) {$T(1,2)$} (-2.8,-1.6);
	\draw[thick, dotted] (-2.7,-1.6) -- (1.7,-1.6);
\end{tikzpicture}
\caption{Tree $T$ with subgraphs $T(1,2)$ and $T_{3}(1,2)$.}\label{FigConcreteTree}
\end{center}
\end{figure}
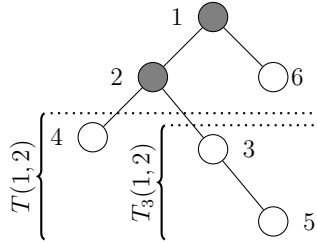
\end{example}

\begin{definition}\cite[Definition 2.3]{HmM}\label{NEBtree}
Let $T$ be a tree on $n$ vertices, and $w$ be a vertex of $T$. $T$ is defined to have \textit{nearly even branching property at $w$} (in short, $T$ is {\it NEB at $w$}) as follows. If $n=1$, $T$ is NEB at $w$.  If $n\geq 2$, $T$ is NEB at $w$ if the following conditions are satisfied:
\begin{itemize}
\item[(i)] $T(w)$ has exactly one odd component if $n$ is even, and $T(w)$ has no odd component if $n$ is odd; and
\item[(ii)] for each neighbor $v$ of $w$ in $T$, $T_v(w)$ is NEB at $v$. 
\end{itemize} 
\end{definition}

\begin{obs}\label{ObstTwoOddatV}
If a tree $T$ is not NEB with respect to a vertex $v$, then there is a vertex $w$ such that $T(w)$ has at least two odd components.
\end{obs}
\begin{proof}
Let $v_1 = v$. If $T(v_1)$ has at least two odd components, then $w=v_1$. Otherwise there are vertices $v_2, \ldots, v_k$ such that $T(v_1,v_2,\ldots, v_k)$ has at least two odd connected components. Let $w=v_k$. Now $T(w)$ has one more branch (at $v_{k-1}$) than $T_{v_k}(v_1,v_2,\ldots, v_{k-1})$, thus it has at least two odd components.
\end{proof}

For a vertex $v$, let $\n(v)$ denote the set of all neighbors of $v$. Let $T$ be a tree which is not NEB at a vertex $v_1$. There exists $v_2, v_3, \ldots, v_k$ such that $T_{v_k}(v_1, v_2, \ldots, v_{k-1})$ is not NEB at $v_k$, but every $T_{w}(v_1, v_2, \ldots, v_k)$ is NEB at $w$ for all $w \in \n(v_k)\setminus \{ v_{k-1}\}$. We call such $T_{v_k}(v_1,\ldots, v_{k-1})$ a minimal non-NEB subtree (with respect to $v_1$). 

\begin{example}
Tree $T$ shown in Figure \ref{FigNontrivialNonNEB} is not NEB at vertex $1$ because $T_3(1,2)$ is not NEB with respect to vertex $3$. But $T_4(1,2,3)$ and $T_5(1,2,3)$ both are NEB with respect to $4$ and $5$, respectively. Hence, $T_3(1,2)$ is a minimal non-NEB subtree of $T$ with respect to vertex $1$.

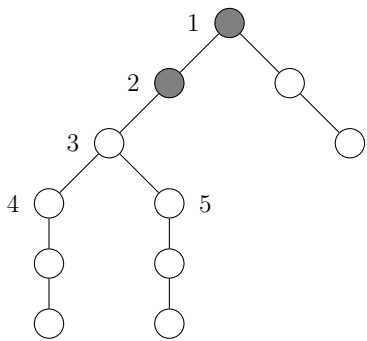
\begin{figure}[h]
\begin{center}
\begin{tikzpicture}[scale=.8]
	\node[circle, fill, gray] () at (0,0) {};
	\node[circle, draw, black] (1) at (0,0) {};
	\node[,scale=.8] () at (-.6,0) {$1$};
	
	\node[circle, fill, gray] () at (-1,-1) {};
	\node[draw, black, circle] (2) at (-1,-1) {};
	\node[,scale=.8] () at (-1.6,-1) {$2$};
	
	\node[draw, circle] (6) at (1,-1) {};
	\node[draw, circle] (7) at (2,-2) {};
	
	\node[draw, circle] (3) at (-2,-2) {};
	\node[,scale=.8] () at (-2.6,-2) {$3$};
	
	\node[draw, circle] (31) at (-3,-3) {};
	\node[,scale=.8] () at (-3.6,-3) {$4$};
	
	\node[draw, circle] (32) at (-1,-3) {};
	\node[,scale=.8] () at (-.4,-3) {$5$};
	
	\node[draw, circle] (41) at (-3,-4) {};
	\node[,scale=.8] () at (-3.6,-4) {};
	
	\node[draw, circle] (42) at (-1,-4) {};
	\node[,scale=.8] () at (-.4,-4) {};
	
	\node[draw, circle] (51) at (-3,-5) {};
	\node[,scale=.8] () at (-3.6,-5) {};
	
	\node[draw, circle] (52) at (-1,-5) {};
	\node[,scale=.8] () at (-.4,-5) {};
	
	\draw[] (51) -- (41) -- (31) -- (3) -- (32) -- (42) -- (52);
	\draw[] (3) -- (2) -- (1) -- (6) -- (7);

\end{tikzpicture}
\caption{Subtree $T_3(1,2)$ is a minimal non-NEB subtree of $T$ with respect to vertex $1$.}\label{FigNontrivialNonNEB}
\end{center}
\end{figure}
\end{example}

The following theorem gives the most important known result we use in this article. It shows that if a tree $T$ is NEB at a vertex, then $T$ has a full matching.

\begin{theorem}\cite[Corollary $5.3$]{HmM}\label{neb}
Let $G$ be  a connected graph on $n$ vertices and $\lambda_1$, $\lambda_2$, \ldots, $\lambda_n$ distinct real numbers such that \[\lambda_j=-\lambda_{n+1-j},\] for all  $j=1,\ldots,n$. If $G$ has a spanning tree which is NEB at a vertex, then $\match(G)=\lfloor \frac{n}{2}\rfloor$ and there exists a matrix $A\in S^-(G)$ with eigenvalues $\i\lambda_1, \ldots, \i\lambda_n$.
\end{theorem}

\section{Characterizations of NEB trees and connected graphs with a perfect matching}
Theorem \ref{neb} shows that if a tree $T$ is NEB at a vertex, then $T$ has a full matching. It is natural to ask if the converse is true. In the next theorem we show that the converse is indeed true.

\begin{theorem}\label{NEBiffFullMatching}
Let $T$ be a tree on $n$ vertices. Tree $T$ is NEB with respect to some vertex $v$ if and only if $\match(T)=\lfloor {\frac{n}{2}} \rfloor$.
\end{theorem}
\begin{proof}
The forward direction is proved in \cite[Observation $3.8$]{HmM}. For the backward direction, assume $T$ is not NEB with respect to any vertex. By Observation \ref{ObstTwoOddatV} there is a vertex $v$ of $T$ such that $T(v)$ has at least two odd components. Let $T_{w_1}(v)$ and $T_{w_2}(v)$ be two such odd components.

There are two cases: \begin{enumerate}
	\item[Case 1:] $n$ is even. Thus, $\lfloor {\frac{n}{2}} \rfloor = {\frac{n}{2}}$, that is, $T$ has a perfect matching, and exactly one of the neighbors of $v$ is matched with $v$. That is, at least one of the $w_1$ or $w_2$ are not matched with $v$. Without loss of generality, assume that $w_1$ is the vertex which is not matched (See Figure \ref{FigEvenNotMatched}). Then $T_{w_1}(v)$ is a tree with odd number of vertices, hence it has a vertex which is not matched. Furthermore, since $T$ has an even number of vertices, it has at least 2 vertices which are not matched. That contradicts the assumption that $T$ has a perfect matching.
	
\begin{figure}[h]
\begin{center}
\begin{tikzpicture}[scale=.7]
	\node[] (35) at (-1,-2.7) {};
	\node[] (36) at (0,-2.7) {};
	
	\node[draw, circle, fill] (44) at (-4,-4.5) {};
	\node[,scale=.8] () at (-4.6,-4.5) {$v$};
	\node[] (46) at (-2,-3.6) {};
	\node[] (47) at (-1,-3.6) {};
	
	\node[draw, thick, circle, dashed, minimum size = 40] (55) at (-5,-7) {};
	\draw[fill] (-4.7,-6.1) circle (.1);
	\draw[fill] (-5.5,-6.6) circle (.1);
	\node[scale=.5] () at (-4.9,-5.8) {$w_1$}; 
	\node[scale=.5] () at (-5,-8.3) {$T_{w_1}(v)$};

	\node[draw, thick, circle, dashed, minimum size = 40] (56) at (-3,-8) {};
	\draw[fill] (-3.3,-7.1) circle (.1);
	\node[scale=.5] () at (-3,-6.8) {$w_2$};
	\node[scale=.5] () at (-2.8,-9.2) {$T_{w_2}(v)$};

	\node[draw, thick, circle, dashed, minimum size = 40] (58) at (-.8,-7.8) {};
	\draw[fill] (-1.4,-7.1) circle (.1);
	\node[scale=.5] () at (-1.4,-6.7) {$w_3$};
	\node[scale=.5] () at (0,-9) {$T_{w_3}(v)$};

	\node[rotate=30] (57) at (.7,-7) {$\cdots$};

	\node[draw, thick, circle, dashed, minimum size = 40] (59) at (2,-6) {};
	\draw[fill] (1.1,-5.7) circle (.1);
	\node[scale=.5] () at (.7,-5.8) {$w_r$};
	\node[scale=.5] () at (2.5,-7.2) {$T_{w_r}(v)$};
	
	\draw[] (44) -- (55);

	\draw[ultra thick] (44) -- (56);
	\draw[] (44) -- (59);
	\draw[] (44) -- (58);
	
	\draw[very thick] (-4.7,-6.1) -- (-5.2,-6.6);
	\begin{scope}[shift={(.4,-.4)}]	\draw[very thick] (-4.7,-6.1) -- (-5.2,-6.6); \end{scope}
	\begin{scope}[shift={(-.2,-1.1)}]	\draw[very thick] (-4.7,-6.1) -- (-5.2,-6.6); \end{scope}
	\begin{scope}[shift={(.4,-1)}]	\draw[very thick] (-4.7,-6.1) -- (-5.2,-6.6); \end{scope}
	\begin{scope}[shift={(-.2,-.6)}]	\draw[very thick] (-4.7,-6.1) -- (-5.2,-6.6); \end{scope}

	\begin{scope}[shift={(2,-1.1)}]
	\draw[very thick] (-4.7,-6.1) -- (-5.2,-6.6);
	\begin{scope}[shift={(.4,-.4)}]	\draw[very thick] (-4.7,-6.1) -- (-5.2,-6.6); \end{scope}
	\begin{scope}[shift={(-.2,-1.1)}]	\draw[very thick] (-4.7,-6.1) -- (-5.2,-6.6); \end{scope}
	\begin{scope}[shift={(.4,-1)}]	\draw[very thick] (-4.7,-6.1) -- (-5.2,-6.6); \end{scope}
	\begin{scope}[shift={(-.2,-.6)}]	\draw[very thick] (-4.7,-6.1) -- (-5.2,-6.6); \end{scope}
	\end{scope}
	
	\begin{scope}[]
	\draw[very thick] (-1.4,-7.1) -- (-1.4,-7.8);
	\begin{scope}[shift={(.4,-.4)}]	\draw[very thick] (-1.4,-7.1) -- (-1.4,-7.8); \end{scope}
	\begin{scope}[shift={(.8,0)}]	\draw[very thick] (-1.4,-7.1) -- (-1.4,-7.8); \end{scope}
	\begin{scope}[shift={(1.2,-.4)}]	\draw[very thick] (-1.4,-7.1) -- (-1.4,-7.8); \end{scope}
	\end{scope}

	\begin{scope}[rotate = 15, shift={(1,1.3)}]
	\draw[very thick] (-1.4,-7.1) -- (-1.4,-7.8);
	\begin{scope}[shift={(.4,-.4)}]	\draw[very thick] (-1.4,-7.1) -- (-1.4,-7.8); \end{scope}
	\begin{scope}[shift={(.8,0)}]	\draw[very thick] (-1.4,-7.1) -- (-1.4,-7.8); \end{scope}
	\begin{scope}[shift={(1.2,-.4)}]	\draw[very thick] (-1.4,-7.1) -- (-1.4,-7.8); \end{scope}
	\end{scope}
\end{tikzpicture}
\caption{Vertex $w_1$ is not matched with vertex $v$ where $T_{w_1}(v)$ is an odd component.}\label{FigEvenNotMatched}
\end{center}
\end{figure}
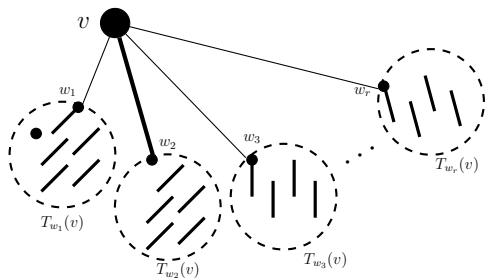

	\item[Case 2:] $n$ is odd. Fix $v_1$ and find a minimal non-NEB subtree of $T$ (with respect to $v_1$), say $T_{v_k}(v_1,\ldots,v_{k-1}) = T'$. Let $v=v_k$. Since $T'$ is a minimal non-NEB subtree of $T$, $T'(v)$ has at least two odd components.
	\begin{itemize}
		\item[(a)] $T'(v)$ has at least 3 odd components, then similar to Case 1, $v$ is matched with at most one of its neighbors in an odd component, and other two odd components each have at least one vertex which is not matched. Hence $\match(T) < \lfloor \frac{n}{2} \rfloor$. 
		\item[(b)] $T'(v)$ has exactly two odd components, say $T'_{w_1}(v)$ and $T'_{w_2}(v)$. Now, consider $T_{v}(w_1)$ (See Figure \ref{FigOddNotMatched}), which has even number of vertices. If $T_{v}(w_1)$ is NEB at $v$, then $T$ is NEB at $w_1$ by minimality of $T'$.  Otherwise, $T_{v}(w_1)$ has at least two vertices which are not matched, by Case 1. Furthermore, since $T$ has odd number of vertices, then it has at least 3 vertices which are not matched. Thus $\match(T) < \lfloor \frac{n}{2} \rfloor$.
\end{itemize}

\end{enumerate}

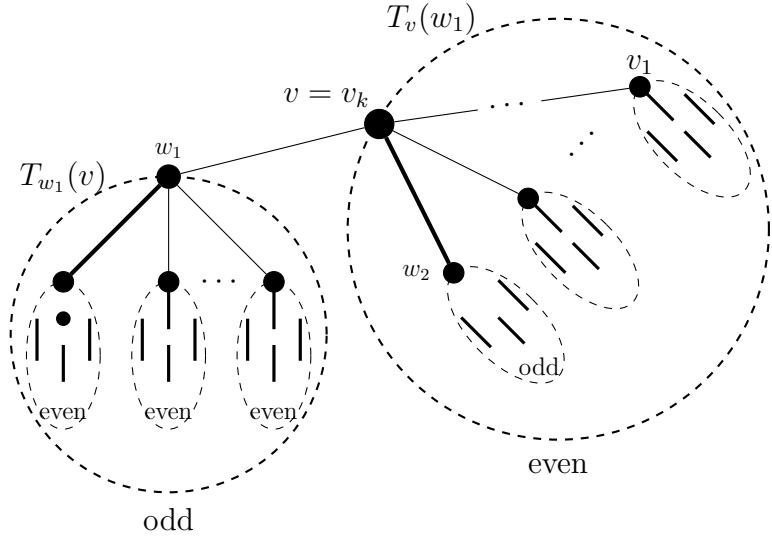
\begin{figure}[h]
\begin{center}
\begin{tikzpicture}[scale=.7]
	\node[draw, circle, fill] (v) at (0,0) {};
	\node[] () at (-1,.5) {$v=v_k$};
	\node[draw, circle, fill, scale=.8] (w1) at (-4,-1) {};
	\node[scale=.8] () at (-4,.-.5) {$w_1$};
	
	\draw[dashed, thick] (3.4,-2) circle (4);
	\node[scale=1] () at (3.4,-6.5) {even}; 
	\draw[dashed, thick] (-4,-4) circle (3);
	\node[scale=1] () at (-4,-7.5) {odd}; 
	
	\draw[] (v) -- (w1);
	
	\node[draw, circle, fill, scale=.7] (w11) at (-6,-3) {};
	\node[draw, circle, fill, scale=.7] (w12) at (-4,-3) {};
	\node[scale =1] () at (-3,-3) {$\cdots$};
	\node[draw, circle, fill, scale=.7] (w13) at (-2,-3) {};
	
	\draw[ultra thick] (w11) -- (w1);
	\draw[] (w12) -- (w1);
	\draw[] (w13) -- (w1);
	
	\node[circle, fill, scale=.5] () at (-6,-3.7) {};
	\draw[very thick] (-6.5,-3.7) -- (-6.5,-4.5);	
	\draw[very thick] (-6,-4.2) -- (-6,-4.9);
	\draw[very thick] (-5.5,-3.7) -- (-5.5,-4.5);

	\begin{scope}[shift={(2,0)}]
	\draw[very thick] (-6.5,-3.7) -- (-6.5,-4.5);	
	\draw[very thick] (-6,-4.2) -- (-6,-4.9);
	\draw[very thick] (-6,-3) -- (-6,-3.9);
	\draw[very thick] (-5.5,-3.7) -- (-5.5,-4.5);
	\end{scope}
	
	\begin{scope}[shift={(4,0)}]
	\draw[very thick] (-6.5,-3.7) -- (-6.5,-4.5);	
	\draw[very thick] (-6,-4.2) -- (-6,-4.9);
	\draw[very thick] (-6,-3) -- (-6,-3.9);
	\draw[very thick] (-5.5,-3.7) -- (-5.5,-4.5);
	\end{scope}
	
	\draw[dashed] (-6,-4.4) ellipse (.7 and 1.4);
	\node[scale=.8] () at (-6,-5.5) {even}; 
	\draw[dashed] (-4,-4.4) ellipse (.7 and 1.4);
	\node[scale=.8] () at (-4,-5.5) {even}; 
	\draw[dashed] (-2,-4.4) ellipse (.7 and 1.4);
	\node[scale=.8] () at (-2,-5.5) {even}; 
	
	\begin{scope}[rotate=45, shift={(5,0)}]
	\node[draw, circle, fill, scale=.7] (w21) at (-6,-3) {};
	\node[draw, circle, fill, scale=.7] (w22) at (-4,-3) {};
	\node[rotate = 45] () at (-2.5,-3) {$\cdots$};
	\node[draw, circle, fill, scale=.7] (w23) at (-1,-3) {};
	
	\node[rotate=6] (v2) at (-3,-1.5) {$\cdots$};
	\node[] () at (-.7,-2.7) {$v_1$};
	\node[scale=.8] () at (-6.5,-2.5) {$w_2$};
	\node[scale=.8] () at (-6.1,-5.4) {odd};
	\draw[ultra thick] (w21) -- (v);
	\draw[] (w22) -- (v);
	\draw[] (w23) -- (v2) -- (v);
	
	\draw[dashed] (-6,-4.4) ellipse (.7 and 1.4);
	\draw[dashed] (-4,-4.4) ellipse (.7 and 1.4);
	\draw[dashed] (-1,-4.4) ellipse (.7 and 1.4);

	\draw[very thick] (-6.5,-3.7) -- (-6.5,-4.5);	
	\draw[very thick] (-6,-4.2) -- (-6,-4.9);
	\draw[very thick] (-5.5,-3.7) -- (-5.5,-4.5);

	\begin{scope}[shift={(2,0)}]
	\draw[very thick] (-6.5,-3.7) -- (-6.5,-4.5);	
	\draw[very thick] (-6,-4.2) -- (-6,-4.9);
	\draw[very thick] (-6,-3) -- (-6,-3.9);
	\draw[very thick] (-5.5,-3.7) -- (-5.5,-4.5);
	\end{scope}
	
	\begin{scope}[shift={(5,0)}]
	\draw[very thick] (-6.5,-3.7) -- (-6.5,-4.5);	
	\draw[very thick] (-6,-4.2) -- (-6,-4.9);
	\draw[very thick] (-6,-3) -- (-6,-3.9);
	\draw[very thick] (-5.5,-3.7) -- (-5.5,-4.5);
	\end{scope}
	\end{scope}
	
	\node[] () at (1,2) {$T_{v}(w_1)$};
	\node[] () at (-6,-1) {$T_{w_1}(v)$};
\end{tikzpicture}
\caption{Tree $T$ and subtrees $T_{v}(w_1)$ and $T_{w_1}(v)$.}\label{FigOddNotMatched}
\end{center}
\end{figure}
\end{proof}

We get the following corollary from Theorem \ref{neb} and Theorem \ref{NEBiffFullMatching}.

\begin{corollary}
Let $T$ be a tree on $n$ vertices. Then $\match(T)= \lfloor \frac{n}{2} \rfloor$ if and only if there is a real skew-symmetric matrix $A$ with distinct eigenvalues whose graph is $T$.
\end{corollary}

Below we mention a rather easy exercise in graph theory, and we will use it to extend the above result to connected graphs.

\begin{lemma}\label{GFulliffTfull}
Let $G$ be a connected graph on $n$ vertices. Then $\match(G)=\lfloor \frac{n}{2} \rfloor$ if and only if $G$ has a spanning tree $T$ with $\match(T)=\lfloor \frac{n}{2} \rfloor$. More specifically, for any matching $M$ (of any size) of $G$, there is a spanning tree $T$ of $G$ which includes all the edges of $M$.
\end{lemma}
\begin{proof}
If a spanning tree of $G$ has a full matching, then $G$ has a full matching. Fix a matching $M$ of $G$. Every cycle of $G$ contains an edge which is not in $M$. Delete one such edge from $G$, and repeat this process with the obtained graph which is still connected, it contains all edges of $M$, and it has at least one less cycle than $G$. The process stops with a connected acyclic graph (tree) on $n$ vertices, since $G$ has finitely many cycles. The obtained graph is a spanning tree of $G$ which contains all edges of $M$.
\end{proof}

\begin{theorem}\label{fullmatching}
Let $G$ be a connected graph on $n$ vertices. If $\match(G)= \lfloor \frac{n}{2} \rfloor$, then for any $n$ distinct real numbers $\lambda_1,\ldots,\lambda_n$ such that $\lambda_j=-\lambda_{n+1-j}$ for all  $j=1,\ldots,n$, there is a matrix $A\in S^{-}(G)$ with eigenvalues $\i\lambda_1,\ldots,\i\lambda_n$. Conversely if there is a matrix $A\in S^{-}(G)$ with distinct eigenvalues, then $\match(G)= \lfloor \frac{n}{2} \rfloor$.
\end{theorem}
\begin{proof}
Assume that $\match(G)= \lfloor \frac{n}{2} \rfloor$. By Lemma \ref{GFulliffTfull} graph $G$ has full matching if and only if it has a spanning tree $T$ with a full matching. Also by Theorem \ref{NEBiffFullMatching}, $T$ has a full matching if and only if $T$ is NEB with respect to a vertex. Thus, by Theorem \ref{neb}, $G$ realizes a real skew-symmetric matrix $A$ with the given eigenvalues. 

Conversely suppose that there is a real skew-symmetric matrix $A$ with distinct eigenvalues whose graph is $G$.   Then, by Lemma \ref{nonzeroeigen} and Theorem \ref{maxskewrank}, $$2 \floor{\frac{n}{2}} = \rank(A) \leq \MR(G) = 2 \match(G).$$ That is, $\floor{ \frac{n}{2}} \leq \match(G)$. Since $\match(G) \leq \floor{ \frac{n}{2}}$ for any graph $G$, we have $\match(G)=\floor{ \frac{n}{2}}$.
\end{proof}

Theorem \ref{fullmatching} immediately implies the following corollary giving a spectral condition for a connected graph to have a perfect matching or a near perfect matching.

\begin{corollary}
Let $G$ be a connected graph. Then $G$ has a full matching if and only if there is a matrix $A\in S^{-}(G)$ with distinct eigenvalues.
\end{corollary}

\section{Spectral characterization of graphs with arbitrary matching number}
It is known that $\match(G)=k$ if and only if $\MR(G)=2k$, i.e., $G$ realizes a skew-symmetric with $2k$ nonzero eigenvalues by Theorem \ref{maxskewrank} and Lemma \ref{nonzeroeigen}. In this section we prove that that these eigenvalues can be any $k$ distinct nonzero purely imaginary numbers and their conjugate pairs. Similar to approaches in \cite{HmM, ms} we are going to use the Jacobian method, so we need to define an appropriate function and show its Jacobian is nonsingular when it is evaluated at some point.

Let $G$ be a graph on $n$ vertices with matching number $k$, and $k+m$ edges where $m>0$. Fix a maximum matching $\mathcal M$ of $G$ and without loss of generality assume $\mathcal M=\{\set{1,2}, \set{3,4}, \ldots, \set{2k-1,2k}\}$. Assume the $m$ edges of $G$ that are not in $\mathcal M$ are of the forms $e_l=\{i_l,j_l\}$, for $l=1,2,\ldots,m$. Let $x_1,\ldots,x_k,y_1,\ldots,y_m$ be $k+m$ independent indeterminates and set 
$$\bm x = (\seq{x}{k}{,}),\text{ and } \bm y = (\seq{y}{m}{,}).$$
We define a skew-symmetric matrix of variables where $x_j$ are in the positions corresponding to the edges in $\mathcal M$, and $y_l$ are in the positions of the edges not in $\mathcal M$.
Let $M=M(\bm x, \bm y)$ be an $n\times n$ skew-symmetric matrix whose $(2j-1,2j)$-entry is $x_j$, $(2j,2j-1)$-entry is $-x_j$, for $j=1,2,\ldots,k$, and for $l=1,2,\ldots,m$ let the $(i_l,j_l)$-entry of $M$ to be $y_l$ where $i_l<j_l$, and $-y_l$, otherwise. Note that Since $\match(G)=k$, $G-\{1,2,\ldots,2k\}$ has no edges.  Thus $M$ has the following form. 
$$M = \left[ \begin{array}{c|c}
N & L\\ \hline\vspace*{-12pt}\\ -L^T & O
\end{array} \right],$$ where $N$ is the upper left $2k\times 2k$ block of $M$, $O$ is the square zero matrix of size $n-2k$, and $L$ contains only $y_l$'s and zeros. Note that $N$ contains zero entries, all of the $x_j$'s, and some or none of $y_l$'s. In particular, the $(2j-1,2j)$-th entry of $N$ is $x_j$, for $j=1,2,\ldots,k$.

\begin{example} Consider the following graph $G$ on 6 vertices with 6 edges and $\match(G)=2$.  
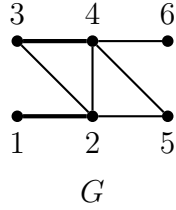
\begin{center}
\begin{tikzpicture}
[colorstyle/.style={circle, draw=black!100,fill=black!100, thick, inner sep=0pt, minimum size=1.25 mm},>=stealth]

\node (1) at (0,0)[colorstyle, label=below:$1$]{};
\node (2) at (1,0)[colorstyle, label=below:$2$]{};
\node (4) at (1,1)[colorstyle, label=above:$4$]{};
\node (3) at (0,1)[colorstyle, label=above:$3$]{};
\node (5) at (2,0)[colorstyle, label=below:$5$]{};
\node (6) at (2,1)[colorstyle, label=above:$6$]{};

\draw [ultra thick] (1)--(2);
\draw [thick] (2)--(5);
\draw [thick] (2)--(3);
\draw [ultra thick] (3)--(4);
\draw [thick] (4)--(6);
\draw [thick] (5)--(4)--(2);

\node at (1,-1)[]{$G$};
\end{tikzpicture}
\end{center}
For the above $G$, $\mathcal M=\{\{1,2\},\{3,4\}\}$ is a maximum matching. So $M=M(\bm x, \bm y)$ would have the following form.
\[M=M(\bm x, \bm y)=\left[\begin{array}{cccc|cc}
0 & \bm{x_1} & 0 & 0 & 0 & 0\\
-\bm{x_1} & 0 & y_1 & y_2 & y_3 & 0\\
0 & -y_1 & 0 & \bm{x_2} & 0 & 0\\
0 & -y_2 & -\bm{x_2} & 0 & y_4 & y_5\\ \hline 
0 & -y_3 & 0 & -y_4 & 0 & 0\\
0 & 0 & 0 & -y_5 & 0 & 0
\end{array} \right].\]
\end{example}

A real evaluation $A$ of $M$ is obtained by assigning real values to indeterminates in $\bm x$ and $\bm y$. Clearly such evaluation $A$ is a skew-symmetric matrix whose graph is a subgraph of $G$ and the eigenvalue of $A$ are purely imaginary occurring in conjugate pairs and some zeros. Define the following ordering of the purely imaginary axis of the complex plane: for two numbers $a$ and $b$ on the imaginary axis of the complex plane let $a \geq b$ if $-a\i \geq -b\i$  and the equality holds if and only if $a=b$. 

Define $F: \mathbb{R}^{k+m} \to \mathbb{R}^n$ by \[ F(\bm x,\bm y) = \big(-\i \lambda_1 (M), -\i \lambda_2 (M), \ldots, -\i \lambda_n (M) \big), \] where $\lambda_j(M)$ is the $j$-th largest eigenvalue of $M$. Note that, some of the middle components of $F$ might be zero. Furthermore, since $\lambda_{j}(M)=-\lambda_{n-j+1}(M)$ for $j=1,\ldots,n$, $F$ is completely defined by half of its components, say the ones in upper half-plane and zeros. Moreover, $M$ has at most $k$ nonzero eigenvalues in the upper half-plane since $\MR(G)=2\match(G)=2k$. That is, $F$ is completely determined by its first $k$ components. 

Define $f: \mathbb{R}^{k+m} \to \mathbb{R}^k$ by \[ f(\bm x,\bm y) = \big(-\i \lambda_1 (M), -\i \lambda_2 (M), \ldots, -\i \lambda_k (M) \big). \]

Let $\lambda_1 > \lambda_2 > \ldots > \lambda_k>0$  be $k$ distinct nonzero purely imaginary numbers. Set $\bm a = (\seq{- \i \, \lambda}{k}{,})\in \mathbb R^k$, $\bm b = (0,\ldots, 0)\in \mathbb R^m$ and $A = M(\bm a, \bm b)$. 
Then $A$ is the block diagonal matrix 
\begin{equation} \label{adirectsummatching}
A = \bigoplus_{j=1}^{k} \left[ \begin{array}{cc}
0 & -\i \lambda_j\\ \i \lambda_j & 0
\end{array} \right] \oplus O_{n-2k}.
\end{equation}
That is, \[A  = \left[ \begin{array}{c|ccc}
 \begin{array}{cc|cc|ccc|cc}
0 & -\i \lambda_1 &0&0& &\cdots&&0 &0  \\
\i \lambda_1 & 0 &0&0& &\cdots&&0 &0 \\\hline
0&0&0& -\i\lambda_2& &\cdots&&0 &0 \\
0&0&\i\lambda_2&0 & &\cdots&&0 &0 \\\hline
 &&& &  && \ddots &&\\
  \vdots&\vdots&\vdots&\vdots &  &\ddots&  &\vdots&\vdots\\
 &&& & \ddots && & & \\\hline
 0&0&0&0 & & \hdots&& 0 & -\i \lambda_k\\
 0&0&0&0 & & \hdots&& \i \lambda_k & 0
\end{array} & & O& \\\hline
 & & & \\
O & & O & \\
 & & &  
\end{array} \right].\]
It easy to check that the nonzero eigenvalues of $A$ are $\pm\lambda_1,\pm \lambda_2,\ldots,\pm\lambda_k$ and consequently $f\at_A =f(\bm a, \bm b)= (\seq{- \i \, \lambda}{k}{,})$. We want to show that the Jacobian of $f$ evaluated at the point $(\bm a, \bm b)$ is nonsingular. It is known that the eigenvalues and eigenvectors of a matrix with distinct eigenvalues are continuous differentiable functions of the entries of the matrix \cite{eigenderivative}. The following lemma shows the derivative of the nonzero eigenvalues of a skew-symmetric matrix with $2k$ distinct nonzero eigenvalues and $n-2k$ zero eigenvalues with respect to the entries of the matrix, in terms of the entries of their corresponding eigenvectors.

\begin{lemma}\label{derivativeofeff}
Let $A$ be an $n\times n$ real skew-symmetric matrix with distinct nonzero eigenvalues $\seq{\lambda}{k}{,}$ in the upper half-plane, and corresponding unit eigenvectors $\seq{\bm v}{k}{,}$. Let $A(t)=A + t E_{rs} - t E_{sr}$, for $t\in (-\varepsilon , \varepsilon)$, where $\varepsilon$ is a small positive number. Also, let $\lambda_j(t)$ be the $j$-th largest eigenvalue of $A(t)$ with corresponding eigenvector $\bm v_j(t)$, and $\bm v_{j_r}$ denote the $r$-th entry of the vector $\bm v_j$. Then
\[ \frac{\d \lambda_j(t)}{\d t} \at_{\, t=0} = 2\i \im( \overline{\bm v_{j_r}} \bm v_{j_s} ),\]  where $\im(z)$ denotes the imaginary part of the complex number $z$.
\end{lemma}
\begin{proof}
Note that $A(t)$, $\lambda_j(t)$ and $\bm v_j(t)$ are continuous functions of $t$, so $A(0) = A$, $\lambda_j (0)=\lambda_j, \bm v_j (0)= \bm v_j$, and when $t \to 0$ we have \[A(t) \to A,\] \[\lambda_j (t) \to \lambda_j,\] \[\bm v_j (t) \to \bm v_j.\] 
Furthermore,  \[\dot{A}(0) = E_{rs} - E_{sr},\] and \[A(t)\bm  v_j(t) = \lambda_j(t)\bm  v_j(t).\] 
Differentiating both sides with respect to $t$ we get \[\dot{A}(t) \bm v_j(t) + A(t) \dot{\bm v_j}(t) = \dot{\lambda_j}(t) \bm v_j(t) + \lambda_j(t) \dot{\bm v_j}(t).\] 
Set $t=0$, then \[(E_{rs} - E_{sr})\bm v_j + A \dot{\bm v_j}(0) = \dot{\lambda_j}(0) \bm v_j + \lambda_j \dot{\bm v_j}(0).\]
Multiplying both sides by $\overline{\bm v_j}^T$ from left we get \[ \overline{\bm v_j}^T (E_{rs} - E_{sr}) \bm v_j +  \overline{\bm v_j}^T A \dot{\bm v_j}(0) = \dot{\lambda_j}(0) \overline{\bm v_j}^T \bm v_j + \lambda_j \overline{\bm v_j}^T \dot{\bm v_j}(0). \]
Since $A$ is skew-symmetric $A \overline{\bm v_j} = -\lambda_j \overline{\bm v_j}$. Hence
\[ \overline{\bm v_j}^T A = (A^T \overline{\bm v_j})^T = (-A \overline{\bm v_j} )^T = (-(- \lambda_j \overline{\bm v_j}))^T=\lambda_j \overline{\bm v_j}^T.\]
Also, \[ \overline{\bm v_j}^T (E_{rs} - E_{sr}) \bm v_j =  \overline{\bm v_{j_r}}v_{j_s}-\overline{\bm v_{j_s}}v_{j_r}= 2 \i \im( \overline{\bm v_{j_r}} \bm v_{j_s} ). \] Thus
\[ 2 \i \im( \overline{\bm v_{j_r}} \bm v_{j_s} ) + \lambda_j \overline{\bm v_j}^T  \dot{\bm v_j}(0) = \dot{\lambda_j}(0) \overline{\bm v_j}^T \bm v_j + \lambda_j \overline{\bm v_j}^T \dot{\bm v_j}(0). \]
The second term in left hand side is equal to the second term in right hand side, and $\bm v_j$'s are unit vectors, that is, $\overline{\bm v_j}^T \bm v_j = 1$. Hence \[  2 \i \im( \overline{\bm v_{j_r}} \bm v_{j_s} ) = \dot{\lambda_j}(0). \]
\end{proof}

\begin{corollary}
For $M$, $A$, and $\lambda_j$'s defined as above, let $r = 2l-1$, $s=2l$, and $x_l$ be the entry in the $(r,s)$ position of $M$. Then we have
\[ \frac{\partial}{\partial x_l}\big(-\i \, \lambda_j(M)\big) \at_A = \begin{cases}
1 \text{, if $j=l$,}\\
0 \text{, otherwise.}
\end{cases} \]
\end{corollary}

\begin{proof}
Note that for $A$ we have \[\bm v_j = \frac{1}{\sqrt{2}} \left[\begin{array}{rrrrrrrr} 0 & \cdots & 0 & \i & -1 & 0 & \cdots & 0\end{array}\right]^T,\] where the nonzero entries are at $2j-1$ and $2j$ positions. Also note that \[  \frac{\partial}{\partial x_l}\big( \, \lambda_j(M)\big) \at_A = \frac{\d \lambda_j(t)}{\d t} \at_{\, t=0}. \]
Then by Lemma \ref{derivativeofeff} 
\begin{align*}
 \frac{\partial}{\partial x_l}\big(-\i \, \lambda_j(M)\big) \at_A &= (-\i) 2 \i \im \big( \, \overline{\bm v_{j_{2l-1}}} \bm v_{j_{2l}}\big) \\ 
 &= \begin{cases}
2 \im ( \, \frac{-\i}{\sqrt{2}} \frac{-1}{\sqrt{2}} ) \text{, if $j=l$,}\\
0 \text{, otherwise.}
\end{cases} \\
&= \begin{cases}
1 \text{, if $j=l$,}\\
0 \text{, otherwise.}
\end{cases} 
\end{align*}
This completes the proof.
\end{proof}

\begin{corollary}\label{JacobianAtA}
For the matrix $A$ and function $f$ defined as above we have
\[\jac(f)\at_A = I_k,\] where $I_k$ denotes the $k\times k$ identity matrix. Hence, $\jac(f)\at_A$ is nonsingular.
\end{corollary}

Now we are ready to prove the main result of this section which characterizes the graphs with matching number $k$. We will use the Implicit Function Theorem, mentioned below. For a full treatment of the topic see \cite{kran02}.

\begin{theorem}[\textbf{Implicit Function Theorem}]\label{IFT}
Let $F: \mathbb{R}^{s+r} \rightarrow \mathbb{R}^s$ be a continuously differentiable
 function on an open subset $U$ of  $\mathbb{R}^{s+r}$ defined by 
 \[F(\bm x,\bm y)=(F_1(\bm x,\bm y), F_2(\bm x,\bm y), \ldots, F_s(\bm x,\bm y)),\]
 where $\bm x=(x_1, \ldots, x_s) \in \mathbb{R}^s$, $\bm y = (y_1, \ldots, y_r) \in \mathbb{R}^r$, and $F_i$'s are real valued multivariate functions. Let $(\bm a,\bm b)$ be an element of $U$
 with $\bm a\in \mathbb{R}^s$ and $\bm b\in \mathbb{R}^r$,  and $\bm c$ be an element of $\mathbb{R}^s$ such that 
  $F(\bm a,\bm b)=\bm c$. If \[ \jac_x(F)\at_{(\bm a,\bm b)} = \left[\frac{\partial F_i}{\partial x_{j}}\at_{(\bm a,\bm b)} \right]_{s\times s} \]
is nonsingular, then there exist an open neighborhood $V$  of $\bm a$ and an open neighborhood $W$ of $\bm b$ such that  $V \times W \subseteq U$ such that for each $\bm y\in W$ there is an $\bm x\in V$ with 
$F(\bm x,\bm y)=\bm c$. Furthermore, for any $(\bar{\bm  a}, \bar{\bm  b}) \in V\times W$ such that $F(\bar{\bm  a}, \bar{\bm  b})=\bm c$, $\jac(F)\at_{(\bar{\bm a},\bar{\bm b})}$ is also nonsingular.
\end{theorem}

\begin{theorem}\label{mathchk}
Let $G$ be a graph on $n$ vertices, and $\lambda_1 > \lambda_2 > \ldots > \lambda_k>0$ be $k$ distinct nonzero purely imaginary numbers where $2k \leq n$. Then $\match(G) = k$ if and only if \begin{enumerate}[(a)]
	\item there is a matrix $A \in S^-(G)$ whose eigenvalues are $\seq{\pm\lambda}{k}{,}$ and $n-2k$ zeros, and
	\item for all matrices $A \in S^-(G)$, $A$ has at most $2k$ nonzero eigenvalues.
\end{enumerate}
\end{theorem}
\begin{proof}
Assume that (a) and (b) hold. Then (a) and Lemma \ref{nonzeroeigen} imply that $\MR(G)\geq 2k$. Furthermore (b) and Lemma \ref{nonzeroeigen} imply that $\MR(G) \leq \rank{A} = 2k$. Thus $\MR(G)=2k$. By Theorem \ref{maxskewrank} we have $\match(G)=\frac{\MR(G)}{2} = \frac{2k}{2}=k$.
 
Now assume that $\match(G)=k$. If $G$ is a disjoint union of edges, then the matrix $A$ given by (\ref{adirectsummatching}) has the desired properties. Assume that there is an edge which is not in a maximum matching, that is, $G$ has $k+m$ edges where $m>0$. Consider the function $f$, and the matrices $M$ and $A$ as above. Note that $f\at_A = (-\i \lambda_1, -\i \lambda_2, \ldots, -\i \lambda_k )$, and $\jac(f)\at_A$ is nonsingular, by Corollary \ref{JacobianAtA}. Then by the Implicit Function Theorem (Theorem \ref{IFT}) there are open sets $U \in \mathbb{R}^k$ and $V \in \mathbb{R}^m$, such that $(-\i \lambda_1,\ldots,-\i \lambda_k) \in U$ and $(0,\ldots,0) \in V$, and for any $(\varepsilon_1, \ldots, \varepsilon_m) \in V$, there is a $(-\i \widehat{\lambda_1}, \ldots, -\i \widehat{\lambda_k}) \in U$ close to $(-\i \lambda_1,\ldots,-\i \lambda_k)$, such that \[f(-\i \widehat{\lambda_1}, \ldots, -\i \widehat{\lambda_k},\varepsilon_1, \ldots, \varepsilon_m) = (-\i \lambda_1, \ldots, -\i  \lambda_k).\] Since $V$ is an open neighborhood of $(0,\ldots,0) \in \mathbb{R}^m$, one can choose all $\varepsilon_i \neq 0$. Let $\widehat{A} = M(-\i \widehat{\lambda_1}, \ldots, -\i \widehat{\lambda_k},\varepsilon_1, \ldots, \varepsilon_m)$. Then eigenvalues of $A$ are $(-\i \lambda_1, -\i \lambda_2, \ldots, -\i \lambda_k )$ and graph of $A$ is $G$. That is $(a)$ holds. Also, by Theorem \ref{maxskewrank} and Lemma \ref{nonzeroeigen}, $(b)$ holds.
\end{proof}

Note that for a given graph $G$ with matching number $k$, there might exist skew-symmetric matrices with less than $2k$ nonzero eigenvalues whose graph is $G$. One easy example is the complete bipartite graph $K_{n,n}, n\geq 2$. The matching number of $K_{n,n}$ is $n$ and its skew-adjacency matrix $A=xy^T-yx^T$, where
$x=\left[\begin{array}{c|c}\all1 & \all1\end{array} \right]^T$ and  $y=\left[\begin{array}{c|c}\all1 & 2\cdot\all1\end{array} \right]^T$ and  $\all1$ is the all ones vector of order $n$, has only two nonzero eigenvalues $\pm n \i$.

\begin{remark}
Theorem \ref{mathchk} shows that the graphs $G$ of order $n$ whose matching number is $k$ are precisely those graphs with the maximum skew rank $2k$ such that for any given set of $k$ distinct nonzero purely imaginary numbers there is a real skew-symmetric matrix $A$ with graph $G$ whose spectrum consists of the given $k$ numbers, their conjugate pairs and $n-2k$ zeros.
\end{remark}

\begin{small}

\end{small}
\end{document}